\pdfoutput=1
\RequirePackage{ifpdf}
\ifpdf 
\documentclass[pdftex]{sigma}
\else
\documentclass{sigma}
\fi

\numberwithin{equation}{section}

\newtheorem{Theorem}{Theorem}[section]
\newtheorem{Corollary}[Theorem]{Corollary}

\newtheorem{Proposition}[Theorem]{Proposition}
 { \theoremstyle{definition}
\newtheorem{Definition}[Theorem]{Definition}

\newtheorem{Remark}[Theorem]{Remark} }

\newcommand{\ay}{{\rm i}}
\newcommand{\Epsilon}{\varepsilon}

\usepackage[all]{xy}
\usepackage{array}
\usepackage{mathrsfs}

\begin{document}

\newcommand{\arXivNumber}{1805.04646}

\renewcommand{\thefootnote}{}

\renewcommand{\PaperNumber}{118}

\FirstPageHeading

\ShortArticleName{Integral Regulators for Higher Chow Complexes}

\ArticleName{Integral Regulators for Higher Chow Complexes\footnote{This paper is a~contribution to the Special Issue on Modular Forms and String Theory in honor of Noriko Yui. The full collection is available at \href{http://www.emis.de/journals/SIGMA/modular-forms.html}{http://www.emis.de/journals/SIGMA/modular-forms.html}}}

\Author{Muxi LI}

\AuthorNameForHeading{M.~Li}

\Address{University of Science and Technology, Hefei, Anhui, P.R.~China}
\Email{\href{mailto:limuxi@ustc.edu.cn}{limuxi@ustc.edu.cn}}

\ArticleDates{Received May 12, 2018, in final form October 31, 2018; Published online November 03, 2018}

\Abstract{Building on Kerr, Lewis and M\"uller-Stach's work on the rational regulator, we prove the existence of an integral regulator on higher Chow complexes and give an explicit expression. This puts firm ground under some earlier results and speculations on the torsion in higher cycle groups by Kerr--Lewis--M\"uller-Stach, Petras, and Kerr--Yang.}

\Keywords{integral regulator; higher Chow groups; algebraic cycles; Abel--Jacobi map}

\Classification{14C15; 14C25; 19F27}

\renewcommand{\thefootnote}{\arabic{footnote}}
\setcounter{footnote}{0}

\section{Introduction}\label{S1}

Higher Chow groups were introduced by S.~Bloch in the mid-80's as a geometric representation of algebraic $K$-theory \cite{Bl0}. For $X$ a smooth quasi-projective variety over an infinite field $k$, Bloch's Grothendieck--Riemann--Roch theorem identifies them \emph{rationally} with certain graded pieces of $K$-theory:
\begin{gather}\label{bgrr}
{\rm CH}^p(X,n)\otimes\mathbb{Q}\simeq {\rm Gr}^P_{\gamma}K^{\rm alg}_nX\otimes\mathbb{Q}.
\end{gather}
As Bloch showed, these groups come with natural Chern class maps
\begin{gather}\label{zreg}
{{\rm AJ}}_{\mathbb{Z}}^{p,n}\colon \ {\rm CH}^p(X,n)\to H_{\mathscr{D}}^{2p-n}(X,\mathbb{Z}(p))
\end{gather}
to the cohomology of the underlying variety \cite{Bl0a}, which ``interpolate'' Griffiths's Abel--Jacobi maps on Chow groups (i.e.,~$K_0$) and Borel's regulators on the higher $K$-theory of number fields.

While abstractly defined, these maps were successfully computed in many specific cases by Bloch, Beilinson, Deninger, and others. However, an explicit general formula only emerged in the work of Kerr, Lewis and M\"{u}ller-Stach \cite{KLM,KL}. By introducing a~subcomplex $Z^p_{\mathbb{R}}(X,\bullet)\overset{\imath}{\hookrightarrow} Z^p(X,\bullet)$ of cycles in good position with respect to the ``wavefront'' set of certain currents on $\big(\mathbb{P}^1\big)^n$, they are able to construct a map of complexes
\begin{gather} \label{klmreg}
\widetilde{{\rm AJ}}\colon \ Z^p_{\mathbb{R}}(X,\bullet)\rightarrow C^{2p-\bullet}_{\mathscr{D}}(X,\mathbb{Z}(p))
\end{gather}
agreeing \emph{rationally} with \eqref{zreg}. While the explicit formula will be recalled in Section~\ref{S4}, we remark that $\mathcal{C}^m_{\mathcal{D}}(X,\mathbb{Z}(p))$ consists of triples $(T,\Omega,R)$ where $T\in \mathcal{C}^m(X,(2\pi\ay)^p\mathbb{Z})$ is a smooth chain of real codimension $m$, and $\Omega\in F^p \mathcal{D}^m(X)$, $R\in\mathcal{D}^{m-1}(X)$ are currents. The map~\eqref{klmreg} sends higher Chow cycles $Z$ to triples $(T_Z,\Omega_Z,R_Z)$, and the relations that make~\eqref{klmreg} a morphism of complexes are $\partial T_Z=T_{\partial Z}$, $\mathrm{d}[\Omega_Z]=\Omega_{\partial Z}$, and $\mathrm{d}[R_Z]=R_{\partial Z}+\Omega_Z-\delta_{T_Z}$ (where $\delta_{T_Z}$ is the current of integration over~$T_Z$).

At first glance, the ``KLM formula''~\eqref{klmreg} looks well-adapted to detecting torsion. For example, if $X=\operatorname{Spec}(k)$ is a point, then a portion of \eqref{klmreg} takes the form
\begin{gather*}
 \xymatrix{\cdots \ar [r] & Z^p_{\mathbb{R}}(k,2p) \ar [r]^{\partial} \ar [d]_{W\mapsto}^{(2\pi\ay)^p W\cap T_{2p}} & Z^p_{\mathbb{R}}(k,2p-1) \ar [r]^{\partial}\ar[d]_{Z\mapsto}^{\frac{1}{(2\pi\ay)^{p-1}}\int_Z R_{2p-1}} & Z^p_{\mathbb{R}}(k,2p-2) \ar [r] \ar[d]^0 & \cdots \\
\cdots \ar [r] & \mathbb{Z}(p) \oplus 0 \oplus 0 \ar @{^(->} [r] & 0 \oplus 0 \oplus \mathbb{C} \ar [r] & 0\oplus 0\oplus 0 \ar [r] & \cdots,}
\end{gather*}
where $T_{2p}=\mathbb{R}_{<0}^{2p} \subset \big(\mathbb{P}^1\big)^{2p}$ and $R_{2p-1}$ is a certain $(2p-2)$-current on $\big(\mathbb{P}^1\big)^{2p-1}$. We want to detect torsion in ${\rm CH}^p(k,2p-1)$ by the middle map; denote the image of $Z\in\ker(\partial)$ by $\mathscr{R}(Z)\in \mathbb{C}/\mathbb{Z}(p)$. In particular, if $Z_1:=\big(1-1/t,1-t,t^{-1}\big)_{t\in \mathbb{P}^1}\in Z^2(\mathbb{Q},3)$, we find that $\mathscr{R}(Z_1)=\pi^2/6\in \mathbb{C}/(2\pi\ay)^2 \mathbb{Z}$, in agreement with the known result that ${\rm CH}^2(\mathbb{Q},3)$ is 24-torsion (see \cite{Pe}).

Unfortunately, it appears very difficult to determine whether $\imath$ is an integral quasi-isomor\-phism, as expected in~\cite{KLM}. Indeed, the proof in~\cite{KL} that this inclusion of complexes is a $\mathbb{Q}$-quasi-isomorphism makes essential use of Kleiman transversality in $K$-theory and hence of some form of~\eqref{bgrr}. So the KLM formula only induces a ``rational regulator''
\begin{gather}\label{klmaj}
{{\rm AJ}}_{\mathbb{Q}}^{p,n}\colon \ {\rm CH}^p(X,n)\to H_{\mathscr{D}}^{2p-n}(X,\mathbb{Q}(p)).
\end{gather}
It is easy to see the problem: we could have that the class of $Z$ in $H_{2p-1}\big(Z^p_{\mathbb{R}}(k,2p-1)\big)$ \emph{and} its $\widetilde{{\rm AJ}}$-image are $m$-torsion (but nonzero), whilst $Z$ is a boundary in the larger complex (hence zero in ${\rm CH}^p(k,2p-1)$). That is, there would be some $\mathcal{W}\in Z^p(k,2p)\setminus Z^p_{\mathbb{R}}(k,2p)$ with $\partial \mathcal{W}=Z$, but only $mZ \in \partial\big( Z^p_{\mathbb{R}}(k,2p)\big)$. Moreover, even if we could improve the result on $\imath$ (and eliminate this particular worry), it would remain inconvenient to find representative cycles in $Z^p_{\mathbb{R}}(X,n)$.

An alternative is to extend KLM to a formula that works on all cycles. Doing this with one map of complexes on $Z^p(X,\bullet)$ is probably too optimistic, as one can't just wish away the ``wavefront sets'' arising from the branch cuts in the $\{\log(z_i)\}$. Our first idea was to try an infinite family of homotopic maps on nested subcomplexes $Z^p_{\varepsilon}(X,\bullet)$ with union $Z^p(X,\bullet)$, by allowing cycles in good position with respect to ``perturbations'' of these branch cuts by sufficiently small nonzero ``phase''~$e^{\ay\epsilon}$, $0<\epsilon<\varepsilon$. Provided one tunes the branches of log in the regulator currents accordingly, and the same $\epsilon$ is used for each $z_i$, one gets a morphism of complexes on the $\varepsilon$-subcomplexes. Since the homotopy class of this morphism is independent of~$\epsilon$, this approach would define an integral refinement of $\widetilde{{\rm AJ}}$ provided the $\varepsilon\to 0$ limit of the ``perturbed'' subcomplexes gives all of~$Z^p(X,\bullet)$. Unfortunately, this is not true: there is a counterexample involving triples of functions on a curve, see Section~\ref{simp}. So a more subtle approach is required.

In particular, we need a way to vary phases $\epsilon_i$ independently for the branches of $\log(z_i)$, so as to place weaker demands on our cycles. But this can never lead to a morphism of complexes from $Z^p(X,\bullet)$, since this independence would conflict with the way the Bloch differential $\partial$ intersects cycles with all the facets. On the other hand, one has an explicit $\mathbb{Z}$-homotopy equivalence for the inclusion $\mathscr{N}^p(X,\bullet) \subset Z^p(X,\bullet)$ of the normalized cycles, on which the differential restricts to just one facet \cite{Bl1}. In $\mathscr{N}^p(X,\bullet)$, we now consider the ``$\Epsilon$-subcomplex'' $\mathscr{N}^p_{\Epsilon}(X,\bullet)$, consisting of cycles which are in good position with respect to the $\big(e^{\ay\epsilon_1},\ldots,e^{\ay\epsilon_n}\big)$-perturbed wavefront set for any $(\epsilon_1,\ldots,\epsilon_n)$ belonging to $B^n_\Epsilon:=\big\{ \underline{\epsilon}\in \mathbb{R}^n \,|\, 0<\epsilon_1<\Epsilon,0<\epsilon_2<e^{-1/\epsilon_1},\ldots,0<\epsilon_n<e^{-1/\epsilon_{n-1}}\big\}$.

Our main technical results are
\begin{Theorem}$\bigcup_{\varepsilon>0}\mathscr{N}^p_{\varepsilon}(X,\bullet)=\mathscr{N}^p(X,\bullet)$.
\end{Theorem}

\begin{Theorem}Given $\underline{\epsilon}$, $\underline{\epsilon}'\in B^N_{\epsilon}$, the corresponding morphisms
\begin{gather*}
R_{\underline{\epsilon}},\,R_{\underline{\epsilon}'}\colon \ \tau_{\leq N}\mathscr{N}_{\Epsilon}^p(X,\bullet)\to C_{\mathscr{D}}^{2p-\bullet}(X,\mathbb{Z}(p))
\end{gather*}
induced by the perturbed KLM currents, are integrally homotopic in degrees~$\bullet<N$.
\end{Theorem}
(Here $\tau_{\leq N}$ truncates the complex above the $N^{\rm th}$ term.) These results are proved in Sections~\ref{S3} and~\ref{htpy}, respectively. It is now easy to deduce that, taken over all $\underline{\epsilon}$, these morphisms induce a map of the form~\eqref{zreg} refining~\eqref{klmaj}, see Section~\ref{sec-zaj}. We conclude by indicating several applications of the KLM formula to torsion in Section~\ref{sec-tors} due to~\cite{KLM}, Petras~\cite{Pe}, Kerr--Yang~\cite{MY} which are now validated by our construction, and indicate future work in this direction.

\begin{Remark}In this article, we are working with analytic Deligne cohomology, which is not the optimal generalization to quasi-projective varieties. It's more work to define the map to absolute Hodge cohomology.
\end{Remark}

\section{Higher Chow cycles}\label{S2}

\subsection{Basic definitions}

Definitions in this section follow \cite{KLM}.
Let $X$ be a smooth quasiprojective algebraic variety over an infinite field $k$. An algebraic cycle on $X$ is a finite linear combination $\Sigma n_V[V]$ of subvarieties $V \subset X$, where $n_V \in \mathbb{Z}$.

We define the algebraic $n$-cube (over $k$) by $\Box^n:=\big(\mathbb{P}^1\backslash\{1\}\big)^n$ with face inclusions $\rho^{\mathfrak{f}}_i\colon \Box^{n-1}\rightarrow\Box^n$ ($\mathfrak{f}\in\{0,\infty\}$) sending $(z_1,\ldots,z_{n-1})$ to $(z_1,\ldots,z_{i-1},\mathfrak{f},z_i,\ldots,z_{n-1})$, and coordinate projections $\pi_i\colon \Box^{n}\rightarrow\Box^{n-1}$ sending $(z_1,\ldots,z_n)$ to $(z_1,\ldots,\hat{z_i},\ldots,z_n)$. We call
\begin{gather*}
\partial\Box^n:=\bigcup_{\substack{i=1,\dots,n\\ \mathfrak{f}=0,\infty}}\big(\rho^{\mathfrak{f}}_i\big)_*\Box^{n-1}
\end{gather*}
the facets of $\Box^n$, and
\begin{gather*}
\partial^k\Box^n:=\bigcup_{\substack{i_1<\cdots<i_k\\ \mathfrak{f}_1,\dots,\mathfrak{f}_k=0,\infty}}\big(\rho^{\mathfrak{f}_1}_{i_1}\big)_*\cdots\big(\rho^{\mathfrak{f}_k}_{i_k}\big)_*\Box^{n-k}
\end{gather*}
the codimension-$k$ subfaces of $\Box^n$.

\begin{Definition}$c^p(X,n)\subset Z^p(X\times\Box^n)$ is the free abelian group on irreducible subvarieties $V\subset X\times\Box^n$ of codimension $p$ such that $V$ meets all faces of $X\times\Box^n$ properly.
\end{Definition}

\begin{Definition}The \emph{degenerate cycles} $d^p(X,n)\subset c^p(X,n)$ are defined as $\sum\limits_{i=1}^n \pi_i^*(c^p(X,n-1))$. Set $Z^p(X,n):=c^p(X,n)/d^p(X,n)$.
\end{Definition}

The Bloch differential
\begin{gather*}
\partial_B:=\sum_{i=1}^n (-1)^{i}\big(\rho_i^{0 *}-\rho_i^{\infty *}\big)\colon \ Z^p(X,n)\rightarrow Z^p(X,n-1)
\end{gather*} makes $Z^p(X,\bullet)$ into a complex, with the higher Chow groups ${\rm CH}^p(X,n)$ given by their homology. For convenience, we shall often use cohomological indexing:
\begin{Definition}${\rm CH}^p(X,n):=H^{-n}\{Z^p(X,-\bullet)\}$.
\end{Definition}

\subsection{A moving lemma} We recall the subcomplex from \cite{KLM}. Henceforth we shall take $k$ to be a subfield of $\mathbb{C}$, so we can consider the complex analytic spaces associated to components of a cycle $Z$. Write $T_{z_i}$ for the (codimention 1) geometric chain $z_i^{-1}(\mathbb{R}_{<0})$, oriented so that $\partial T_{z_i}=(z_i)=z_i^{-1}(0)-z_i^{-1}(\infty)$. Let $c^p_{\mathbb{R}}(X,n)$ be the set of all the cycles $Z\in c^p(X,n)$ whose components (or rather, their analytizations) intersect $X\times(T_{z_1}\cap\cdots\cap T_{z_i})$ and $X\times\big(T_{z_1}\cap\cdots\cap T_{z_i}\cap\partial^k\Box^n\big)$ properly for all $1\leq i\leq n$ and $1\leq k<n$, and $d^p_{\mathbb{R}}(X,n):=c^p_{\mathbb{R}}(X,n)\cap d^p(X,n)$. We get a new complex $Z^p_{\mathbb{R}}(X,n):=c^p_{\mathbb{R}}(X,n)/d^p_{\mathbb{R}}(X,n)$. It is shown in~\cite{KL} that this subcomplex is $\mathbb{Q}$-quasi-isomorphic to the original one:

\begin{Theorem}[Kerr--Lewis]$Z^p_{\mathbb{R}}(X,\bullet)\xrightarrow{\simeq} Z^p(X,\bullet)$.
\end{Theorem}

\subsection{Normalized cycles}\label{nor}

Higher Chow groups may also be computed by complexes of cycles that have trivial boundary on all but one face.
\begin{Definition}$\mathscr{N}^p(X,n):=\{Z\in Z^p(X,n)\,|\,\partial^{\infty}_i Z=0$ for $i<n$, $\partial^0_j Z=0$ for any $j\}$, where $\partial_i^{\mathfrak{f}} Z:=\big(\rho^{\mathfrak{f}}_i\big)^* Z \in Z^p(X,n-1)$.
\end{Definition}
In this section, we will write down an explicit retraction of $Z^p(X,\bullet)$ onto the normalized cycle complex which is homotopic to the identity. The construction is derived from Bloch's manuscript \cite{Bl1}, by replacing the notations from $\big(\mathbb{A}^1\big)^n$ (using $\{0,1\}$ as boundary) by $\big(\mathbb{P}^1\setminus \{1\}\big)^n$ (using $\{0,\infty\}$ as boundary). In addition, Bloch uses a different definition for the normalized cycles: $\mathscr{N}'^p(X,\bullet):=\{Z\in Z^p(X,n)\,|\,\partial^{\infty}_iZ=0$ for $i>1$, $\partial^0_j Z=0$ for any $j\}$; so we need to apply a ``conjugation'' to the proof in \cite{Bl1} as well.

Define $Z^p_{\infty,i}(X,\bullet)=\{Z\in c^p(X,\bullet)\,|\,\partial^{\infty}_jZ=0$ for $j<n-i$, $\partial^0_kZ=0$ for any~$k\}$. We have $Z^p_{\infty,0}(X,\bullet)=\mathscr{N}^p(X,\bullet)$, and inclusions of complexes $Z^p_{\infty,0}(X,\bullet)\subseteq Z^p_{\infty,1}(X,\bullet)\subseteq \cdots$ which stabilize to $Z^p(X,\bullet)$ in any degree. More precisely, we have $Z^p_{\infty,i}(X,n) = Z^p(X,n)$ for $i\geq n$, since $d^p(X,n)\cap Z^p_{\infty,i}(X,n)=\{0\}$.

\begin{Theorem}The inclusion $\mathscr{N}^p(X,\bullet) \subset Z^p(X,\bullet)$ is an integral quasi-isomorphism.
\end{Theorem}

\begin{proof}Any $Z\in Z^p(X,n)$ may be lifted to $c^p(X,n)$, and we may add degenerate cycles to any element of $c^p(X,n)$ to force it into $Z^p_{\infty,n}(X,n)$. The (well-defined) map given by this process is an isomorphism, and we shall tacitly equate $Z^p(X,n)$ and $Z^p_{\infty,n}(X,n)$ in what follows.

For each integer $l \leq n-1$, define $h^l\colon \square^{n+1}\rightarrow\square^n$ by
\begin{gather*}
h^l(z_1,\ldots,z_{n+1}):=\big(z_1,\ldots,z_{l},\tfrac{z_{l+1}z_{l+2}}{z_{l+1}+z_{l+2}-1},z_{l+3},\ldots,z_{n+1}\big)
\end{gather*}
and for $Z\in Z^p(X,n)$, define $H^l(Z):=(-1)^{l}\big(h^l\big)^{-1}(Z)\in Z^p(X,n+1)$. If $l\geq n$, set $H^l(Z)=0$. The map
\begin{gather*}
\phi:=\cdots\big(\mathrm{Id}-\big(\partial\circ H^l+H^l\circ \partial\big)\big)\circ(\mathrm{Id}-\big(\partial\circ H^{l-1}+H^{l-1}\circ \partial\big)\big)\circ \\
\hphantom{\phi:=}{} \cdots \circ\big(\mathrm{Id}-\big(\partial\circ H^0+H^0\circ \partial\big)\big)
\end{gather*}
stabilizes in any degree and so defines an endomorphism $\phi\colon Z^p(X,\bullet)\rightarrow Z^p(X,\bullet)$, which is visibly homotopic to the identity.

To determine its image, write (for any $Z\in Z^p(X,n)$)
\begin{gather*}
 \big(\partial \circ H^l\big)Z=\sum_{k=1}^{n-l-1}(-1)^{n-l+k+1}\partial_{n-k+1}^{\infty}Z\big(z_1,\ldots,\tfrac{z_{l+1}z_{l+2}}{z_{l+1}+z_{l+2}-1},\ldots,z_n\big) \\
\hphantom{\big(\partial \circ H^l\big)Z=}{}+\sum_{k=n-l+1}^{n}(-1)^{n-l+k}\partial_{n-k+1}^{\infty}Z\big(z_1,\ldots,\tfrac{z_{l}z_{l+1}}{z_{l}+z_{l+1}-1},\ldots,z_n\big)
\end{gather*}
and
\begin{gather*}
 \big(H^l\circ \partial\big)Z=\sum_{k=1}^{n}(-1)^{n-l+k}\partial_{n-k+1}^{\infty}Z\big(z_1,\ldots,\tfrac{z_{l+1}z_{l+2}}{z_{l+1}+z_{l+2}-1},\ldots,z_n\big),
\end{gather*}
where the notation means that we pull back (the equations defining) $\partial^{\infty}_{n-k+1}Z$ via $h^l$ (or $h^{l-1}$): $\square^n \to \square^{n-1}$. Thus we have
\begin{gather*}
 \big(\partial\circ H^l+H^l\circ \partial\big)Z=\sum_{k=n-l+1}^{n}(-1)^{n-l+k}\partial_{n-k+1}^{\infty}Z\big(z_1,\ldots,\tfrac{z_{l}z_{l+1}}{z_{l}+z_{l+1}-1},\ldots,z_n\big) \\
\hphantom{\big(\partial\circ H^l+H^l\circ \partial\big)Z=}{}+\sum_{k=n-l}^{n}(-1)^{n-l+k}\partial_{n-k+1}^{\infty}Z\big(z_1,\ldots,\tfrac{z_{l+1}z_{l+2}}{z_{l+1}+z_{l+2}-1},\ldots,z_n\big).
\end{gather*}
In particular, for $Z\in Z^p_{\infty,i}(X,\bullet)$, we have $\big(d\circ H^l+H^l\circ d\big)Z=0$ for $l \leq n-i-2$. For $l=n-i-1$, we find
\begin{gather*}
 Z':=Z-\big(\partial\circ H^l+H^l\circ\partial\big)Z = Z-\partial^{\infty}_{n-i}Z\big(z_1,\ldots,\tfrac{z_{n-i}z_{n-i+1}}{z_{n-i}+z_{n-i+1}-1},\ldots,z_n\big),
\end{gather*}
which belongs to $Z^p_{\infty,i-1}$. Applying $\big(\mathrm{Id}-\big(\partial\circ H^{l+1}+H^{l+1}\circ\partial\big)\big)$ then maps $Z'$ to some $Z''\in Z^p_{\infty,i-2}$, and so forth until finally we reach $Z^p_{\infty,0}(X,n)=\mathscr{N}^p(X,n)$. Since all the $\partial\circ H^l + H^l\circ \partial$ are zero on $Z^p_{\infty,0}$, $\phi|_{\mathscr{N}}$ gives the identity on normalized cycles.

We have thus constructed a morphism $\phi\colon Z^p(X,\bullet)\to \mathscr{N}^p(X,\bullet)$, whose composition with the inclusion $\mathscr{N}^p\hookrightarrow Z^p$ is homotopic to (resp. equal to) the identity on $Z^p$ (resp. $\mathscr{N}^p$); thus $\phi$ and the inclusion are both quasi-isomorphisms.
\end{proof}

As explicit expressions for $\phi$ in low dimension, we have
\begin{gather*}\phi(Z(z_1,z_2))=Z(z_1,z_2)-(\partial^{\infty}_1Z)\big(\tfrac{z_1z_2}{z_1+z_2-1}\big),\end{gather*}
and
\begin{gather*}
 \phi(Z(z_1,z_2,z_3))=Z(z_1,z_2,z_3)-(\partial^{\infty}_2Z)\big(z_1,\tfrac{z_2z_3}{z_2+z_3-1}\big)- (\partial^{\infty}_1Z)\big(\tfrac{z_1z_2}{z_1+z_2-1},z_3\big) \\
\hphantom{\phi(Z(z_1,z_2,z_3))=}{}+(\partial^{\infty}_1Z)\big(z_1,\tfrac{z_2z_3}{z_2+z_3-1}\big).
\end{gather*}

\section{Simple perturbations}\label{simp}

The Kerr--Lewis moving lemma can only yield a rational regulator due to the passage through $K$-theory in the proof. Instead, one might consider maps of complexes on a nested family of subcomplexes of $Z^p_{\mathbb{R}}(X,\bullet)$, given by ``perturbing'' the conditions defining $Z_{\mathbb{R}}^p(X,\bullet)$. Though this turns out to be too naive, it is the first step toward a strategy that works.

Begin by defining $Z^p_{\Epsilon}(X,\bullet)$ to be the subcomplex of $Z^p(X,n)$ given by the cycles that intersect $X\times\big(T_{z_1}^{\epsilon}\cap\cdots\cap T_{z_i}^{\epsilon}\big)$ and $X\times\big(T_{z_1}^{\epsilon}\cap\cdots\cap T_{z_i}^{\epsilon}\cap\partial^k\Box^n\big)$ properly for all $1\leq j\leq n$, $1\leq k<n$ and $0<\epsilon<\Epsilon$. Here $T_z^{\epsilon}:=T_{e^{i\epsilon}z}$ is given by $\arg(z)=\pi-\epsilon$, the ``perturbation'' of the branch cut of $\log(z)$ in the currents defined below.

In order for this nested family of subcomplexes to be any better than $Z^p_{\mathbb{R}}(X,\bullet)$, we must have that their union gives us the original $Z^p$:
\begin{gather*}\bigcup_{\Epsilon} Z^p_{\Epsilon}(X,\bullet)=Z^p(X,\bullet).\end{gather*}

Unfortunately, this fails in a very simple case:

\begin{Proposition} \label{counterex}For $X=\operatorname{Spec}(\mathbb{Q}(\ay))$, we have $\bigcup_{\Epsilon} Z^2_{\Epsilon}(X,3) \subsetneq Z^2(X,3).$
\end{Proposition}
\begin{proof}Let $F(z)=\ay z-1$, $G(z)=-\tfrac{(1+z)(1+3z)}{(1+\ay z)(1-2z)}$, and $H(z)=\tfrac{\ay z-1}{3+z}$. Then we have $Z=(F(z),G(z),H(z))_{z\in\mathbb{P}^1}\in Z^2(\text{pt},3)$; but for all $\Epsilon>0$, $Z \notin Z^2_{\Epsilon}(\text{pt},3)$. More precisely, for any $\epsilon>0$, we have $\dim_{\mathbb{R}}\big(Z\cap T^{\epsilon}_{z_1}\cap T^{\epsilon}_{z_2}\cap T^{\epsilon}_{z_3}\big)=0$, not $-1$ (i.e., empty) as required for a proper-analytic intersection. To see this, we need to find a value of $z$ for which $\mathrm{arg}(F)$, $\mathrm{arg}(G)$, $\mathrm{arg}(H)$ equal to $\pi-\epsilon$. Such a value is given by $z=\tan(\epsilon)$.
\end{proof}

Thus we need to find another way to do the ``perturbation'', which will be given in the next section.

\section{Multiple perturbations}\label{S3}
In order to have $Z^{\rm an}$ meet the deformations of $\{T_{z_i}\}$ (and their intersections) properly -- say, for an example like that in the above proof -- we clearly need to make use of the extra degrees of freedom allowed by perturbing each ``branch-cut phase'' independently. For convenience, we shall use the multi-index notation $\underline{\epsilon}:= (\epsilon_1,\ldots ,\epsilon_n )$ in what follows.

Now we are thinking of $T_{z_i}^{\epsilon_i}$ as the location of the jump in the 0-current $\log(z_i)$; these 0-currents will appear in the definition of the regulator-currents $R_Z^{\underline{\epsilon}}$ appearing in the next section. To use these currents to define Abel--Jacobi maps, we will need them to induce morphisms of complexes from a subcomplex of $Z^p(X,\bullet)$ to $C_{\mathscr{D}}^{2p-\bullet}(X,\mathbb{Z}(p))$. Unfortunately, if $Z$ has boundaries at more than one facet of~$\square^n$, say $\partial_1 Z=\big(\rho_1^0\big)^*Z$ and $\partial_2 Z=\big(\rho_2^0\big)^* Z$, the residue terms in $d\big[R_Z^{(\epsilon_1,\ldots,\epsilon_n)}\big]$ will take the form $R_{\partial_1 Z}^{(\epsilon_2,\ldots , \epsilon_n)}$ resp.~$R_{\partial_2 Z}^{(\epsilon_1,\epsilon_3,\ldots ,\epsilon_n)}$. This clearly conflicts with having $D\big(T^{\underline{\epsilon}}_Z,\Omega_Z,R^{\underline{\epsilon}}_Z\big)=\big(T^{\underline{\epsilon'}}_{\partial Z},\Omega_{\partial Z},R^{\underline{\epsilon'}}_{\partial Z}\big)$ for a single choice of $\underline{\epsilon'}$, so we shall need to restrict to the normalized cycles $\mathscr{N}^p(X,\bullet)$ defined in Section~\ref{nor}.

For $\Epsilon>0$, define $B_{\Epsilon}$ as the set of infinite sequences $(\epsilon_1,\epsilon_2,\ldots)$ satisfying
\begin{gather}\label{be}
0<\epsilon_1<\Epsilon,\quad 0<\epsilon_2<\exp(-1/\epsilon_1),\quad 0<\epsilon_3<\exp(-1/\epsilon_2),\quad \ldots,
\end{gather}
and define $B^n_{\Epsilon}$ to comprise the $n$-tuples $\underline{\epsilon}$ satisfying~\eqref{be}.

\begin{Definition}$\mathscr{N}^p_{\Epsilon}(X,\bullet):=\{Z\in\mathscr{N}^p(X,\bullet)\,|\,Z$ intersects $X\times \big(T^{\epsilon_1}_{z_1}\cap\cdots\cap T^{\epsilon_i}_{z_i}\big)$ and $X\times \big(T^{\epsilon_1}_{z_1}\cap\cdots\cap T^{\epsilon_i}_{z_i}\cap\partial^k\Box^n\big)$ properly $\forall\, i,k,\underline{\epsilon}\in B^{\bullet}_{\Epsilon}\}$.
\end{Definition}
\begin{Theorem}$\bigcup_{\varepsilon}\mathscr{N}^p_{\Epsilon}(X,\bullet)=\mathscr{N}^p(X,\bullet)$.
\end{Theorem}
\begin{proof}Consider the projection $(\mathbb{C}^*)^n\rightarrow\big(S^1\big)^n \cong (\mathbb{R}/2\pi\mathbb{Z})^n$ defined by $\big(r_1e^{\ay\epsilon_1},\dots,r_ne^{\ay\epsilon_n}\big)\mapsto(\epsilon_1,\dots,\epsilon_n)$, whose fibers are $T^{\epsilon_1}_{z_1}\cap\cdots\cap T^{\epsilon_n}_{z_n}$. There is also a natural $2^n:1$ map $\big(S^1\big)^n\rightarrow\big(\mathbb{P}_{\mathbb{R}}^1\big)^n$ by taking slopes: $(\epsilon_1,\dots,\epsilon_n)\mapsto(\tan\epsilon_1,\dots,\tan\epsilon_n)$. The composite map $\Theta^n \colon (\mathbb{C}^*)^n\rightarrow\big(S^1\big)^n\rightarrow\big(\mathbb{P}_{\mathbb{R}}^1\big)^n$ is real algebraic, sending $(x_1+\ay y_1,\dots,x_n+\ay y_n)\mapsto(y_1/x_1,\dots, y_n/x_n)$.

Now let $Z\in\mathscr{N}^p(X,n)$ be given. Set $Z^*:=\bar{Z}\cap(X\times (\mathbb{C}^*)^n)$, and let $\tilde{Z^*}$ be its resolution of singularities. The intersections of $Z^*$ with the fibers of $\Theta^n_X\colon X\times(\mathbb{C}^*)^n\rightarrow X\times \big(\mathbb{P}_{\mathbb{R}}^1\big)^n$ are ${Z^*}\cap\big(X\times \big\{T^{-\epsilon_1}_{z_1}\cap\cdots\cap T^{-\epsilon_n}_{z_n}\big\}\big)$. Write $\Theta_Z$ for the composition of $\tilde{Z}^*\to X\times (\mathbb{C}^*)^n$ with $\Theta^n_X$. The set of~$\underline{\epsilon}$ for which these intersections are good is the complement of the non-flat locus $\Delta \subset\big(\mathbb{P}^1_{\mathbb{R}}\big)^n$ of $\Theta_Z$ (see~\cite{No}). Since the flat locus of an algebraic map is Zariski open, $\Delta \subset \big(\mathbb{P}^1_{\mathbb{R}}\big)^n$ is a real subvariety, which is proper by dimension considerations. (That is to say, if all the fibers had real dimension $>2\dim _{\mathbb{C}}(Z^*)-n$, then $Z^*$ would have real dimension $>2\dim _{\mathbb{C}}(Z^*)$, which is a~contradiction.)

Therefore the preimage $\tilde{\Delta}$ of $\Delta$ in $\big(S^1\big)^n$ is real analytic. By the form of the inequalities in~$B_{\Epsilon}$, we know that we can choose an~$\Epsilon>0$ such that $B^n_{\Epsilon}\cap\tilde{\Delta}=\varnothing$. (This follows from \cite[Theorem~1]{LR} for~$\tilde{\Delta}$, and the fact that all derivatives of $e^{-1/x}$ limit to $0$ at~0.) This means that $Z$ intersects $X\times \big(T^{\epsilon_1}_{z_1}\cap\cdots\cap T^{\epsilon_n}_{z_n}\big)$ properly $\forall\,\underline{\epsilon}\in B^{\bullet}_{\Epsilon}$, as desired.

Repeating the argument for $X\times(\mathbb{C}^*)^i\times\big(\mathbb{P}_{\mathbb{C}}^1\big)^{n-i}$ and $X\times(\mathbb{C}^*)^i\times(\{0,\infty\})^k\times\big(\mathbb{P}_{\mathbb{C}}^1\big)^{n-i-k}$, we pick the minimum of the required values of $\Epsilon$, so that $Z$ intersects $X\times \big(T^{\epsilon_1}_{z_1}\cap\cdots\cap T^{\epsilon_i}_{z_i}\big)$ and $X\times \big(T^{\epsilon_1}_{z_1}\cap\cdots\cap T^{\epsilon_i}_{z_i}\cap\partial^k\Box^n\big)$ properly $\forall\, i,k,\underline{\epsilon}\in B^{\bullet}_{\Epsilon}$, which means $Z\in \mathscr{N}^p_{\Epsilon}(X,n)$.
\end{proof}

\section{Abel--Jacobi maps}\label{S4}

In this section, we'll use the strategy in \cite{KLM} to define the Abel--Jacobi maps on our subcomplexes.

\subsection{Definition of Deligne cohomology}

The Deligne cohomology group $H^{2p+n}_{\mathscr{D}}(X,\mathbb{Z}(p))$ is given by the $n^{\text{th}}$ cohomology of the complex
\begin{gather*}
\mathcal{C}^{\bullet+2p}_{\mathcal{D}}(X,\mathbb{Z}(p)):=\{\mathcal{C}^{2p+\bullet}(X,\mathbb{Z}(p))\oplus F^p\mathcal{D}^{\bullet}(X)\oplus\mathcal{D}^{\bullet-1}(X)\}
\end{gather*}
with differential $D$ taking $(a,b,c)\mapsto(-\partial a,-\mathrm{d}[b],\mathrm{d}[c]-b+\delta_a)$, where $\delta_a$ denotes the current of integration over the chain $a$. Here $\mathcal{D}^k(X)$ denotes currents of degree $k$ on $X^{\rm an}$ and $\mathcal{C}^{k}(X,\mathbb{Z}(k))$ denotes $C^{\infty}$ (co)chains of real codimension $k$ and $\mathbb{Z}(k)=(2\pi\ay)^k \mathbb{Z}$ coefficients.

The cup product in Deligne cohomology is defined on the chain level by
\begin{gather*}(a,b,c)\cup(A,B,C):=\big(a\cap A,b\wedge B,c\wedge B + (-1)^{\deg(a)}\delta_a \cdot C \big).\end{gather*}
It becomes commutative upon passage to cohomology. (See \cite{We} for a commutative chain-level construction.) Note that
\begin{gather*}
D[(a,b,c)\cup(A,B,C)]=D(a,b,c)\cup(A,B,C)\; + \;(-1)^{\deg(a)}(a,b,c)\cup D(A,B,C).
\end{gather*}

\subsection{KLM currents}

Firstly we'll review the currents given in \cite{KLM}.
The currents on $\Box^n$ are given by $T_n:=T_{z_1}\cap T_{z_2}\cap \cdots \cap T_{z_n}$, $\Omega_n=\frac{\mathrm{d}z_1}{z_1}\wedge\frac{\mathrm{d}z_2}{z_2}\wedge \cdots \wedge\frac{\mathrm{d}z_n}{z_n}$, and \begin{gather*}R_n=\sum^n_{k=1}(-2\pi\ay)^{k-1}\delta_{T_{z_1} \cap T_{z_2}\cap \cdots \cap T_{z_{k-1}}}\log{z_k}\tfrac{\mathrm{d}z_{k+1}}{z_{k+1}} \wedge \cdots \wedge \tfrac{\mathrm{d}z_n}{z_n}.\end{gather*}
For currents on $X$ associated to a given $Z\in Z^p_{\mathbb{R}}(X,n)$, let $\pi_1\colon \tilde{Z}\rightarrow\Box^n$ and $\pi_2\colon \tilde{Z}\rightarrow X$ be the projections (where $\tilde{Z}$ is a desingularization). Then we have:
\begin{gather*}\widetilde{{\rm AJ}}^{p,n}_{\rm KLM}(Z):=(-2\pi\ay)^{p-n}(\pi_2)_*(\pi_1)^*((2\pi\ay)^nT_n,\Omega_n,R_n).\end{gather*}

\subsection[Currents on $\mathscr{N}^p_{\Epsilon}(X,n)$]{Currents on $\boldsymbol{\mathscr{N}^p_{\Epsilon}(X,n)}$}
Using a similar strategy, for a normalized precycle $Z\in \mathscr{N}^p_{\Epsilon}(X,n)$ and $\underline{\epsilon}\in B^n_{\Epsilon}$, we send
\begin{gather}\label{5.3eq}
Z\mapsto(-2\pi\ay)^{p-n}(\pi_2)_*(\pi_1)^*((2\pi\ay)^nT^{\underline{\epsilon}}_n,\Omega_n,R^{\underline{\epsilon}}_n)=:\mathcal{R}^{n,\underline{\epsilon}}_{\Epsilon}(Z),
\end{gather}
where $T^{\underline{\epsilon}}_n=T_{z_1}^{\epsilon_{1}}\cap T_{z_2}^{\epsilon_{2}}\cap \cdots \cap T_{z_n}^{\epsilon_{n}}$, $\Omega_Z=\frac{\mathrm{d}z_1}{z_1}\wedge\frac{\mathrm{d}z_2}{z_2}\wedge \cdots \wedge\frac{\mathrm{d}z_n}{z_n}$, and \begin{gather*}R^{\underline{\epsilon}}_n=\sum^n_{k=1}(-1)^{{k}\choose{2}}(2\pi\ay)^{k}\delta_{T_{z_1}^{\epsilon_{1}}\cap T_{z_2}^{\epsilon_{2}}\cap \cdots \cap T_{z_{k-1}}^{\epsilon_{k-1}}}\log^{\epsilon_k}{z_k} \tfrac{\mathrm{d}z_{k+1}}{z_{k+1}} \wedge \cdots \wedge \tfrac{\mathrm{d}z_n}{z_n}.\end{gather*} Here $\log^{\epsilon}(z)$ is the branch of $\log$ with argument in $(-\pi-\epsilon,\pi-\epsilon]$. This is a discontinuous function with cut at $T^{\epsilon}_z$, so that ${\rm d}[\log^{\epsilon}(z)]=\frac{{\rm d}z}{z}-2\pi\ay \delta_{T_z^{\epsilon}}$.

The formula \eqref{5.3eq} induces a map of complexes
\begin{gather}\label{map}
\mathcal{R}^{\bullet,\underline{\epsilon}}_{\Epsilon}\colon \ \mathscr{N}^p_{\Epsilon}(X,-\bullet)\rightarrow\mathcal{C}^{2p+\bullet}_{\mathcal{D}}(X,\mathbb{Z}(p)).
\end{gather}
\begin{Proposition}$\mathcal{R}^{\bullet,\underline{\epsilon}}_{\Epsilon}$ is a morphism of complexes.
\end{Proposition}

Therefore we get for each $p$, $n$, $\Epsilon$, and $\underline{\epsilon}\in B_{\Epsilon}$ Abel--Jacobi maps (induced by these maps of complexes)
\begin{gather*}{\rm AJ}^{p,n,\underline{\epsilon}}_{\Epsilon}\colon \ H_n\big(\mathscr{N}^p_{\Epsilon}(X,\bullet)\big)\rightarrow H^{2p-n}_{\mathscr{D}}(X,\mathbb{Z}(p)) .\end{gather*}

\section{Homotopies of Abel--Jacobi maps}\label{htpy}

\subsection{Notations}
Put
\begin{gather*}\mathcal{R}_{z_i} := \big(2\pi\ay T_{z_i}, \tfrac{{\rm d}z_i}{z_i}, \log(z_i)\big), \qquad
 \mathcal{R}^{\epsilon}_{z_i} := \big(2\pi\ay T_{\arg(z_i)=\pi-\epsilon}, \tfrac{{\rm d}z_i}{z_i}, \log^\epsilon(z_i)\big),\end{gather*}
where $\log^\epsilon(z_i)$ is taking branch cut at $\arg(z_i)=\pi-\epsilon$. We write $T_{\arg(z_i)=\pi-\epsilon}$ as $T^{\epsilon}_{z_i}$. Define
\begin{gather*}\mathcal{S}^{\epsilon,\epsilon'}_{z_i}:=\big({-}2\pi\ay \theta^{\epsilon,\epsilon'}_{z_i},0,0\big),\end{gather*}
where $\theta^{\epsilon,\epsilon'}_{z_i}:=\pm\delta_{\{- \arg (z_i)\in(\epsilon,\epsilon')\} }$ are 0-currents. (The sign is positive if $\epsilon>\epsilon'$, negative otherwise.) Clearly we have $D\mathcal{S}^{\epsilon,\epsilon'}_{z_i}=\mathcal{R}^{\epsilon}_{z_i}-\mathcal{R}^{\epsilon'}_{z_i}.$

\subsection{Homotopy property}

In this subsection we will prove the

\begin{Theorem}\label{homothm}For any $\underline{\epsilon}, \underline{\epsilon}'\in B^N_{\Epsilon}$, $\mathcal{R}_{\Epsilon}^{\bullet,\underline{\epsilon}}$ and $\mathcal{R}_{\Epsilon}^{\bullet,\underline{\epsilon}'}$ are $($integrally$)$ homotopic morphisms of complexes in degrees~$\bullet<N$.
\end{Theorem}

For a fixed $N\in \mathbb{N}$, consider $\big(\mathbb{P}^1\big)^N$ with subsets
\begin{gather*}Y_{I,f}:=\cap_{i\in I}\{z_i=f(i)\}\cong\big(\mathbb{P}^1\big)^{N-|I|},\end{gather*}
where $I$ denotes subsets of $\{1,\dots,N\}$ (with complement $I^c$) and $f\colon I\rightarrow\{0,\infty\}$ ranges over the~$2^{|I|}$ possible functions. Write $\rho^i_{f(i)}$ for the inclusion of $Y_{I \{i\},f|_{I \{i\}}}$ in $Y_{I,f}$, $\operatorname{sgn}_I(i)=|\{i'\in I\,|\,i'\leq i\}|$, $\operatorname{sgn}(0)=0$, and $\operatorname{sgn}(\infty)=1$. Consider the double complex
\begin{gather*}E^{a,b}:=\bigoplus_{\substack{I,f\\ |I|=N-a}}C^{2a+b}_{\mathscr{D}}(Y_{I,f}) ,\end{gather*}
with differentials $\delta\colon E^{a,b}\rightarrow E^{a+1,b}$ and $D\colon E^{a,b}\rightarrow E^{a,b+1}$ given by
\begin{gather*} \delta=2\pi\ay\sum _{i\in I}(-1)^{\operatorname{sgn}_I(i)+\operatorname{sgn}(f(i))}\big(\rho^{i}_{f(i)}\big)_*\end{gather*} and the direct sum of Deligne differentials, respectively. (The $Y_{I,f}$'s in the $a\mathrm{th}$ column are $\cong \big(\mathbb{P}^1\big)^a$, and called ``a-faces''.) Put $\mathbb{D}=D+(-1)^b\delta$ for the differential on the associated simple complex $s^k E^{\bullet,\bullet}:=\bigoplus_{a+b=k}E^{a,b}$.

Fix an $n\in\{0,1,\dots,N\}$. For each subset $J=\{j_1,\dots,j_n\}\subset\{1,\dots,N\}$ and $f\colon J^c\rightarrow\{0,\infty\}$, we have an element
\begin{gather*}\mathcal{R}^{\underline{\epsilon}_J}_n:=((2\pi\ay)^nT^{\underline{\epsilon}_J}_n,\Omega_n,R^{\underline{\epsilon}_J}_n)\in C_{\mathscr{D}}^n(Y_{J^c,f}),\end{gather*}
where $\underline{\epsilon}_J=\{\epsilon_{j_1},\dots,\epsilon_{j_n}\}$. Taken together, these yield an $ {{N}\choose{n} }\cdot2^{N-n}$-tuple $\mathcal{R}^{\underline{\epsilon}}_{\Box,n}:=\big\{\mathcal{R}^{\underline{\epsilon}_{I^c}}_n\big\}_{\substack{I,f\\|I|=N-n}}$ $\in E^{n,-n}$. Define
\begin{gather*}\mathcal{R}^{\underline{\epsilon}}_{\Box}:=\big(\mathcal{R}^{\underline{\epsilon}}_{\Box,0},\dots,\mathcal{R}^{\underline{\epsilon}}_{\Box,n}\big)\in\bigoplus^N_{n=0}E^{n,-n}=s^0E^{\bullet,\bullet}.\end{gather*}

\begin{Proposition}$\mathcal{R}^{\underline{\epsilon}}_{\Box}$ is a $0$-cocycle in the simple complex.
\end{Proposition}

\begin{proof}According to \cite[equations~(5.2), (5.3), (5.4)]{KLM}, generally we have (for $\mathcal{R}^{\underline{\epsilon}}_{\Box,n} \in (C^n_{\mathscr{D}}(Y_{I,f}))$, without the loss of generality we take $I=\{1,\dots,n\}^c$):
\begin{gather*}
 D\mathcal{R}^{\underline{\epsilon}}_n=\Bigg({-}(2\pi\ay)^n\sum_{k=1}^n(-1)^k\big(\big(\rho^0_k\big)_*T^{\{\epsilon_1,\dots,\hat{\epsilon_k},\dots,\epsilon_n\}}_{n-1}-\big(\rho^{\infty}_k\big)_*
 T^{\{\epsilon_1,\dots,\hat{\epsilon_k},\dots,\epsilon_n\}}_{n-1}\big), \\
\hphantom{D\mathcal{R}^{\underline{\epsilon}}_n=}{}-2\pi\ay\sum_{k=1}^n(-1)^k\Omega(z_1,\dots,\hat{z_k},\dots,z_n)\delta_{(z_k)}, \\
\hphantom{D\mathcal{R}^{\underline{\epsilon}}_n=}{}-2\pi\ay\sum_{k=1}^n(-1)^kR^{\{\epsilon_1,\dots,\hat{\epsilon_k},\dots,\epsilon_n\}}(z_1,\dots,\hat{z_k},\dots,z_n)\delta_{(z_k)}\Bigg) \\
\hphantom{D\mathcal{R}^{\underline{\epsilon}}_n}{}=-(-1)^{n-1}\delta\big(\big\{ \mathcal{R}^{\{\epsilon_1,\dots,\hat{\epsilon_k},\dots,\epsilon_n\}}_{n-1}\big\}_{k=1,\ldots,n;\, f(k)=0,\infty} \big) ,
\end{gather*}
where for the $\delta$ in the last line, we only consider the components mapping into $C^n_{\mathscr{D}}(Y_{I,f})$. This tells us that \begin{gather*}D\mathcal{R}^{\underline{\epsilon}}_n+(-1)^{n-1}\delta\big(\big\{ \mathcal{R}^{\{\epsilon_1,\dots,\hat{\epsilon_k},\dots,\epsilon_n\}}_{n-1}\big\}_{k=1,\ldots,n;\, f(k)=0,\infty} \big) =0\end{gather*} for any $n$; thus each component of $\mathbb{D}\mathcal{R}^{\underline{\epsilon}}_{\Box}$ is 0, and so $\mathcal{R}^{\underline{\epsilon}}_{\Box}\in\operatorname{Ker}(\mathbb{D})$ is a 0-cocycle.
\end{proof}

\begin{Remark}While wedge products and Leibniz formulas for the extension derivative are not generally valid for currents, they are valid in the setting of exterior products (which is what we use here), see \cite[Appendix~B]{We}.
\end{Remark}

For $\underline{\epsilon}$, $\underline{\epsilon'}$, consider the following $(-1)$-cochain in $E^{\bullet,\bullet}$:
\begin{gather*} \mathcal{S}^{\underline{\epsilon},\underline{\epsilon'}}_{\Box}:= \left\{\mathcal{S}^{\hat{\underline{\epsilon}},\hat{\underline{\epsilon'}}}:=\sum^n_{k=1}(-1)^{k-1}\mathcal{R}^{\epsilon_{m_1}}_{z_1}\cup\cdots \cup\mathcal{R}^{\epsilon_{m_{k-1}}}_{z_{k-1}}\cup\mathcal{S}^{\epsilon_{m_k},\epsilon'_{m_k}}_{z_k}\cup\mathcal{R}^{\epsilon'_{m_{k+1}}}_{z_{k+1}}\cup\cdots \cup\mathcal{R}^{\epsilon'_{m_n}}_{z_n}\!\right\}_{n,\underline{m},\underline{i}}\!.\end{gather*}
It satisfies the following key property:
\begin{Proposition} \label{homoprop}$\mathbb{D}\mathcal{S}^{\underline{\epsilon},\underline{\epsilon'}}_{\Box}=\mathcal{R}^{\underline{\epsilon}}_{\Box}-\mathcal{R}^{\underline{\epsilon'}}_{\Box}$.
\end{Proposition}

\begin{proof}On any given $n$-``face'' ($\cong \big(\mathbb{P}^1\big)^n$) we have
\begin{gather*}
 D\mathcal{S}^{\underline{\epsilon},\underline{\epsilon'}}_n=\sum_{k=1}^n\sum_{l=1}^{k-1}(-1)^{l-1}(-1)^{k-1}\mathcal{R}^{\epsilon_1}_{z_1}\cup\cdots\cup D\mathcal{R}^{\epsilon_l}_{z_l}\cup\cdots\cup\mathcal{S}^{\epsilon_{k},\epsilon'_{k}}_{z_k}\cup\cdots\cup\mathcal{R}^{\epsilon'_n}_{z_n} \\
\hphantom{D\mathcal{S}^{\underline{\epsilon},\underline{\epsilon'}}_n=}{}+\sum_{k=1}^n\sum_{l=k+1}^n(-1)^{l}(-1)^{k-1}\mathcal{R}^{\epsilon_1}_{z_1}\cup\cdots\cup \mathcal{S}^{\epsilon_{k},\epsilon'_{k}}_{z_k}\cup\cdots\cup D\mathcal{R}^{\epsilon'_l}_{z_l}\cup\cdots\cup\mathcal{R}^{\epsilon'_n}_{z_n}\\
\hphantom{D\mathcal{S}^{\underline{\epsilon},\underline{\epsilon'}}_n=}{} +\sum_{k=1}^n\mathcal{R}^{\epsilon_1}_{z_1}\cup\cdots\cup D\mathcal{S}^{\epsilon_k,\epsilon'_k}_{z_k}\cup\cdots\cup\mathcal{R}^{\epsilon'_n}_{z_n}.
\end{gather*}

Noting that $D\mathcal{R}^{\epsilon}_{z_i}=-2\pi\ay(\delta_{(z_i)},\delta_{(z_i)},0)=:-2\pi\ay \Delta_{(z_i)}$ (which commutes with other triples) and $D\mathcal{S}^{\epsilon,\epsilon'}_{z_i}=\mathcal{R}^{\epsilon}_{z_i}-\mathcal{R}^{\epsilon'}_{z_i}$, we can rewrite this expression by applying the telescoping method and rearranging the order of the summation. Denoting $(-1)^{k-1}_l := \begin{cases} (-1)^{k-1}, & l>k, \\ (-1)^{k-2}, & l<k, \end{cases}$
\begin{gather*}
 D\mathcal{S}^{\underline{\epsilon},\underline{\epsilon'}}_n=2\pi\ay\sum_{k=1}^n\sum_{l=1,\, l\neq k}^n(-1)^{l}(-1)^{k-1}_l\Delta_{(z_l)}\mathcal{R}^{\epsilon_1}_{z_1}\cup\cdots\cup\mathcal{S}^{\epsilon_{k},\epsilon'_{k}}_{z_k}\cup\cdots\cup\mathcal{R}^{\epsilon'_n}_{z_n} \\
\hphantom{D\mathcal{S}^{\underline{\epsilon},\underline{\epsilon'}}_n=}{} \text{(with the } l^{\rm th}\text{ term omitted, either before } k \text{ or after } k\text{)} \\
\hphantom{D\mathcal{S}^{\underline{\epsilon},\underline{\epsilon'}}_n=}{} +\sum_{k=1}^n\mathcal{R}^{\epsilon_1}_{z_1}\cup\cdots\cup \big(\mathcal{R}^{\epsilon_k}_{z_k}-\mathcal{R}^{\epsilon'_k}_{z_k}\big)\cup\cdots\cup\mathcal{R}^{\epsilon'_n}_{z_n} \\
\hphantom{D\mathcal{S}^{\underline{\epsilon},\underline{\epsilon'}}_n}{} =2\pi\ay\sum_{l=1}^n(-1)^{l}\Delta_{(z_l)}\sum_{k=1,\, k\neq l}^n(-1)^{k-1}_l\mathcal{R}^{\epsilon_1}_{z_1}\cup\cdots\cup\mathcal{S}^{\epsilon_{k},\epsilon'_{k}}_{z_k}\cup\cdots\cup\mathcal{R}^{\epsilon'_n}_{z_n}
 +\mathcal{R}^{\underline{\epsilon}}_n-\mathcal{R}^{\underline{\epsilon'}}_n \\
\hphantom{D\mathcal{S}^{\underline{\epsilon},\underline{\epsilon'}}_n}{} =-(-1)^n\delta\big( \big\{\mathcal{S}^{\{\epsilon_1,\dots,\hat{\epsilon_l},\dots,\epsilon_n\},\{\epsilon'_1,\dots,\hat{\epsilon'_l},\dots,\epsilon'_n\}}_{n-1} \big\}_{l=1,\ldots,n; f(l)=0,\infty} \big)+\mathcal{R}^{\underline{\epsilon}}_n-\mathcal{R}^{\underline{\epsilon'}}_n.
\end{gather*}

This tells us
\begin{gather}\label{eq64}
D\mathcal{S}^{\underline{\epsilon},\underline{\epsilon'}}_n+(-1)^n\delta\big( \big\{\mathcal{S}^{\{\epsilon_1,\dots,\hat{\epsilon_l},\dots,\epsilon_n\},\{\epsilon'_1,\dots,\hat{\epsilon'_l},\dots,\epsilon'_n\}}_{n-1} \big\}_{l=1,\ldots,n; f(l)=0,\infty} \big)=\mathcal{R}^{\underline{\epsilon}}_n-\mathcal{R}^{\underline{\epsilon'}}_n
\end{gather}
holds on each $n$-face of $\big(\mathbb{P}^1\big)^n$. Thus $\mathbb{D}\mathcal{S}^{\underline{\epsilon},\underline{\epsilon'}}_{\Box}=\mathcal{R}^{\underline{\epsilon}}_{\Box}-\mathcal{R}^{\underline{\epsilon'}}_{\Box}$ holds.
\end{proof}

\begin{proof}[Proof of Theorem~\ref{homothm}]For $Z\in\mathscr{N}^p_{\epsilon}(X,n)$, $\partial Z$ only lives in $X\times\partial^{\infty}_n\Box^n$ by definition. Thus in \eqref{eq64}, the only relevant term in the braces is $\mathcal{S}_{n-1}^{\{\epsilon_1,\ldots,\epsilon_{n-1}\},\{\epsilon_1 ',\ldots,\epsilon_{n-1} '\}}$. Therefore Proposition~\ref{homoprop} implies at once that $\mathcal{R}^{n,\underline{\epsilon}}_{\varepsilon} (Z)- \mathcal{R}^{n,\underline{\epsilon '}}_{\varepsilon}(Z) = D \mathcal{S}^{\underline{\epsilon},\underline{\epsilon'}}_{\Epsilon}(Z) + \mathcal{S}^{\underline{\epsilon},\underline{\epsilon'}}_{\Epsilon}(\partial Z)$, so that
$\mathcal{R}_{\Epsilon}^{\bullet,\underline{\epsilon}}\sim \mathcal{R}_{\Epsilon}^{\bullet,\underline{\epsilon}'}$ as clai\-med.
\end{proof}

\section{The integral Abel--Jacobi map}\label{sec-zaj}

Recall our map of complexes from $\eqref{map}$, with $n^{\text{th}}$ term
\begin{gather*}\mathcal{R}^{n,\underline{\epsilon}}_{\Epsilon}\colon \ \mathscr{N}^p_{\Epsilon}(X,n)\rightarrow C^{2p-n}_{\mathscr{D}}(X,\mathbb{Z}(p)).\end{gather*}

According to our result from the last section, we know that for $\underline{\epsilon},\underline{\epsilon}'\in B^N_{\Epsilon}$, $\mathcal{R}^{n,\underline{\epsilon}}_{\Epsilon}\sim\mathcal{R}^{n,\underline{\epsilon}'}_{\Epsilon}$; that is to say, they induce the same homomorphism after taking cohomology:
\begin{Corollary}
All the $\underline{\epsilon}\in B_{\Epsilon}$ induce the same map:
\begin{gather*}{\rm AJ}_{\Epsilon}^{p,n}\colon \ H_n\big(\mathscr{N}^p_{\Epsilon}(X,\bullet)\big)\rightarrow H^{2p-n}_{\mathscr{D}}(X,\mathbb{Z}(p)) .\end{gather*}
\end{Corollary}
Moreover, for $\Epsilon'<\Epsilon$ and $\underline{\epsilon}\in B_{\Epsilon'}\subset B_{\Epsilon}$, the following diagram commutes:
\begin{gather*}
 \xymatrix{\mathscr{N}^p_{\Epsilon}(X,\bullet)\ar@{^{(}->}[rr]^{\imath}\ar[rd]_{\mathcal{R}^{\bullet,\underline{\epsilon}}_{\Epsilon}}&&\mathscr{N}^p_{\Epsilon'}(X,\bullet)\ar[ld]^{\mathcal{R}^{\bullet,\underline{\epsilon}}_{\Epsilon'}}\\
 &C^{2p-\bullet}_{\mathscr{D}}(X,\mathbb{Z}(p)),& }
\end{gather*}
which is straightforward from the definition. By taking homology, we have that the following diagram commutes as well:
\begin{gather*}
 \xymatrix{ H_n\big(\mathscr{N}^p_{\Epsilon}(X,\bullet)\big)\ar[rr]^{[\imath]}\ar[rd]_{{\rm AJ}^{p,n}_{\Epsilon}}&&H_n\big(\mathscr{N}^p_{\Epsilon'}(X,\bullet)\big)\ar[ld]^{{\rm AJ}^{p,n}_{\Epsilon'}}\\
 &H^{2p-n}_{\mathscr{D}}(X,\mathbb{Z}(p)).& }
\end{gather*}
In order to get the integral Abel--Jacobi map, we need the following result:
\begin{Theorem}
${\rm CH}^p(X,n)\cong \lim\limits_{\overrightarrow{\Epsilon}} H_n\big(\mathscr{N}^p_{\Epsilon}(X,\bullet)\big)$.
\end{Theorem}

\begin{proof}Since $\mathscr{N}^p_{\Epsilon}(X,\bullet)\subset\mathscr{N}^p(X,\bullet)$, we have $H_n\big(\mathscr{N}^p_{\Epsilon}(X,\bullet)\big)$ maps to $H_n(\mathscr{N}^P(X,\bullet))=$ \linebreak ${\rm CH}^P(X,n)$ for every $\Epsilon$, hence there exists a natural map $\lim\limits_{\overrightarrow{\Epsilon}}H_n\big(\mathscr{N}^p_{\Epsilon}(X,\bullet)\big)\rightarrow {\rm CH}^p(X,n)$. Since $\bigcup\mathscr{N}^p_{\Epsilon}(X,\bullet)=\mathscr{N}^p(X,\bullet)$, this map is surjective. To show it is injective, consider $\xi\in {\rm CH}^p(X,n)$, and $\tilde{\xi}$, $\tilde{\xi}'$ be two representations of $\xi$ in the following sequence:
\begin{gather*}H_n\big(\mathscr{N}^p_{\Epsilon}(X,\bullet)\big)\rightarrow H_n\big(\mathscr{N}^p_{\Epsilon'}(X,\bullet)\big)\rightarrow\cdots\rightarrow {\rm CH}^p(X,n).\end{gather*}
We need to show that $\tilde{\xi}$ and $\tilde{\xi}'$ will eventually merge at some $\epsilon$, that is to say, $\bigcup\partial\mathscr{N}^p_{\Epsilon}(X,n+1)=\partial\mathscr{N}^p(X,n+1)$, which directly comes from the property of normalized cycle and $\bigcup\mathscr{N}^p_{\Epsilon}(X,\bullet)=\mathscr{N}^p(X,\bullet)$.
\end{proof}

Thus we have a well-defined map
\begin{gather*}{\rm AJ}^{p,n}_{\mathbb{Z}}\colon \ {\rm CH}^p(X,n)\rightarrow H^{2p-n}_\mathscr{D}(X,\mathbb{Z}(p))\end{gather*}
given by ${\rm AJ}^{p,n}_{\mathbb{Z}}:=\lim\limits_{\overrightarrow{\Epsilon}}{\rm AJ}^{p,n}_{\Epsilon}$. Precisely, for $Z\in {\rm CH}^p(X,n)$, if $\tilde{Z}\in\operatorname{Ker}(\partial)\subset\mathscr{N}^p_{\Epsilon}(X,n)$ is any choice of class mapping to $Z$ and $\underline{\epsilon}$ any choice of element of $B_{\varepsilon}$, ${\rm AJ}^{p,n}_{\mathbb{Z}}(Z):=\mathcal{R}^{n,\underline{\epsilon}}_{\Epsilon}\big(\tilde{Z}\big)$ is independent of the choices. Thus we have an explicit expression for the integral Abel--Jacobi map:
\begin{gather*}{\rm AJ}^{p,n}_{\mathbb{Z}}(Z)=\lim_{\underline{\epsilon}\rightarrow\underline{0}}\mathcal{R}^{n,\underline{\epsilon}}_{\Epsilon}\big(\tilde{Z}\big).\end{gather*}
Moreover, for $\tilde{Z}$ a representative in $ Z^p_{\mathbb{R}}(X,n)\cap\mathscr{N}^p(X,n)$, we know that $\tilde{Z}$ lies in $\mathscr{N}^p_{\Epsilon}$ for some $\Epsilon>0$, and \begin{gather*}\lim_{\underline{\epsilon}\rightarrow\underline{0}}\mathcal{R}^{n,\underline{\epsilon}}_{\Epsilon}\big(\tilde{Z}\big)=\mathcal{R}\big(\tilde{Z}\big).\end{gather*}
since we have the same map of the level of cohomology for every~$\Epsilon$. In particular, this means that \emph{on cycles belonging to} $Z^p_{\mathbb{R}}(X,n)\cap\mathscr{N}^p(X,n)$, \emph{our integral AJ map is given by the KLM formula.}

\section{Application to torsion cycles}\label{sec-tors}

Recent work of Kerr and Yang \cite{MY} provides explicit representatives for generators of ${\rm CH}^n(\operatorname{Spec}(k)$, $2n-1)$ where $k$ is an abelian extension of $\mathbb{Q}$. We'll check that when $n=2,3$, the cycle given by~\cite{MY} satisfies the normal and proper intersection condition thus belongs to $Z^p_{\mathbb{R}}(X,2p-1)\cap\mathscr{N}^p(X,2p-1)$. For $n=4$, this is taken up in~\cite{KLi}; while for $n\geq5$ finding a normalized generator is a future task.

Let $\xi_N$ be an $N^{\text{th}}$ root of 1.

\begin{Proposition}\label{81} The cycles given in {\rm \cite[equations~(4.1) and~(4.2)]{MY}} lie in $Z^n_\mathbb{R}(\mathbb{Q}(\xi_N),2n-1)\cap\mathscr{N}^n(\mathbb{Q}(\xi_N),2n-1)$.
\end{Proposition}

The $Z^n_\mathbb{R}$ part is given in \cite[Remark~3.3]{MY}. The $\mathscr{N}^n$ part is visible from the boundary computations in \cite[Sections~4.1 and~4.2]{MY}. See \cite[Section~5]{KLi} for torsion calculations arising from Proposition~\ref{81}.

This puts some earlier results on firm ground as well, such as O.~Petras's result in~\cite{Pe} that
\begin{gather*}Z:=\big(1-1/t,1-t,t^{-1}\big)+\big(1-\xi_5/t,1-t,t^{-5}\big)+\big(1-\bar{\xi_5}/t,1-t,t^{-5}\big)\end{gather*}
generates ${\rm CH}^2\big(\mathbb{Q}\big(\sqrt{5},3\big)\big)$ and (since we have $\mathscr{R}(Z)=\operatorname{Li}_2(1)+5(\operatorname{Li}_2(\xi_5)+\operatorname{Li}_2(\bar{\xi_5}))=7\pi^2/30$) is 120-torsion.

\subsection*{Acknowledgements}

This work was supported by the National Science Foundation [DMS-1361147; PI: Matt Kerr]. The author would like to thank his advisor Matt Kerr for great help and discussions, J.~McCarthy for graciously supplying the counterexample in Section~\ref{simp}, J.~Lewis for his interest in this work, and the referee for their great help on improving the exposition.

\pdfbookmark[1]{References}{ref}
\LastPageEnding

\end{document}